\documentclass[a4paper,reqno]{amsart}

\usepackage{amssymb}
\usepackage{amstext}
\usepackage{amsmath}
\usepackage{amscd}
\usepackage{amsthm}
\usepackage{amsfonts}
\usepackage{enumerate}
\usepackage{graphicx}
\usepackage{latexsym}
\usepackage{mathrsfs}
\usepackage{mathtools}
\usepackage[all]{xy}
\xyoption{all}

\usepackage{pstricks}
\usepackage{lscape}
\usepackage{comment}
\usepackage[pagewise]{lineno}

\newtheorem{theorem}{Theorem}[section]

\newtheorem{corollary}[theorem]{Corollary}
\newtheorem{lemma}[theorem]{Lemma}
\newtheorem{proposition}[theorem]{Proposition}
\newtheorem{definition-proposition}[theorem]{Definition-Proposition}

\theoremstyle{definition}
\newtheorem{definition}[theorem]{Definition}

\newtheorem{example}[theorem]{Example}

\newcommand{\Ext}{\operatorname{Ext}\nolimits}

\newcommand{\Hom}{\operatorname{Hom}\nolimits}
\newcommand{\End}{\operatorname{End}\nolimits}

\newcommand{\ind}{\mathsf{ind}\hspace{.01in}}
\renewcommand{\mod}{\mathsf{mod}\hspace{.01in}}

\newcommand{\add}{\mathsf{add}\hspace{.01in}}
\newcommand{\Fac}{\mathsf{Fac}\hspace{.01in}}
\newcommand{\tilt}{\mbox{\rm tilt}\hspace{.01in}}
\newcommand{\sttilt}{\mbox{\rm s$\tau$-tilt}\hspace{.01in}}

\renewcommand{\AA}{\mathcal{A}}
\newcommand{\CC}{\mathcal{C}}

\newcommand{\XX}{\mathcal{X}}
\newcommand{\YY}{\mathcal{Y}}
\newcommand{\PP}{\mathcal{P}}

\newcommand{\rad}{\operatorname{rad}\nolimits}

\newcommand{\RHom}{\mathbf{R}\strut\kern-.2em\operatorname{Hom}\nolimits}

\numberwithin{equation}{section}

\usepackage{paralist}

\def\Im{\mathop{\rm Im}\nolimits}
\def\Ker{\mathop{\rm Ker}\nolimits}
\def\Coker{\mathop{\rm Coker}\nolimits}

\def\rad{\mathop{\rm rad}\nolimits}

\def\id{\mathop{\rm id}\nolimits}
\def\pd{\mathop{\rm pd}\nolimits}

\begin{document}
\title{Tilting modules over Auslander-Gorenstein Algebras}

 \thanks{The first author is supported
by JSPS Grant-in-Aid for Scientific Research (B) 24340004, (C)
23540045 and (S) 15H05738. The second author is supported by NSFC
(Nos. 11401488, 11571164 and 11671174), Jiangsu Government
Scholarship for Overseas Studies (JS-2014-352) and NSF for Jiangsu
Province (BK20130983)}.

\author{Osamu Iyama}
\address{O. Iyama: Graduate School of Mathematics, Nagoya
University, Nagoya, 464-8602, Japan}
\email{iyama@math.nagoya-u.ac.jp}
\author{Xiaojin Zhang}
\address{X. Zhang: School of Mathematics and Statistics, NUIST,
Nanjing, 210044, P. R. China}
\email{xjzhang@nuist.edu.cn}
\maketitle

\begin{abstract}

For a finite dimensional algebra $\Lambda$ and a non-negative
integer $n$, we characterize when the set $\tilt_n\Lambda$ of
additive equivalence classes of tilting modules with projective
dimension at most $n$ has a minimal (or equivalently, minimum)
element. This generalize results of Happel-Unger. Moreover, for an
$n$-Gorenstein algebra $\Lambda$ with $n\geq 1$, we construct a
minimal element in $\tilt_{n}\Lambda$.
 As a result, we give equivalent conditions for a
$k$-Gorenstein algebra to be Iwanaga-Gorenstein. Moreover, for an
$1$-Gorenstein algebra $\Lambda$ and its factor algebra
$\Gamma=\Lambda/(e)$, we show that there is a bijection between
$\tilt_1\Lambda$ and the set $\sttilt\Gamma$ of isomorphism classes
of basic support $\tau$-tilting $\Gamma$-modules, where $e$ is an
idempotent such that $e\Lambda $ is the additive generator of
projective-injective $\Lambda$-modules.
\end{abstract}


\section{Introduction}

Tilting theory is very essential in the representation theory of
algebras. There are many works (see \cite{AsSS, AnHK, Ha} and
references there) which made the theory fruitful. One interesting
topic in tilting theory is to classify tilting modules for some
given algebras. Among these, tilting modules over algebras of large
dominant dimension have gained much more attention. For more
details, we refer to \cite{CX,CrS,NRTZ,IZ,PS,K}.

For an algebra $\Lambda$, denote by $\mod\Lambda$ the category of
finitely generated right $\Lambda$-modules. Recall that a
$\Lambda$-module $T$ in $\mod\Lambda$ is called a {\it tilting
module} of finite projective dimension if the projective dimension
of $T$ is $n<\infty$, $\Ext_{\Lambda}^{i}(T,T)=0$ holds for $i\geq1$
and there is an exact sequence $0\rightarrow\Lambda\rightarrow
T_0\rightarrow\cdots\rightarrow T_n\rightarrow0$ with $T_i\in\add
T$, where we use $\add T$ to denote the subcategory of $\mod\Lambda$
consisting of direct summands of finite direct sums of $T$. We say
that $M,N\in\mod\Lambda$ are \emph{additively equivalent} if $\add
M=\add N$. For a non-negative integer $n$, let $\tilt_n\Lambda$ be
the set consisting of additive equivalence classes of tilting
modules with projective dimension at most $n$, and
$\tilt\Lambda=\tilt_\infty\Lambda:=\bigcup_{n\ge0}\tilt_n\Lambda$.
There is a natural partial order on the set $\tilt\Lambda$ defined
as follows \cite{RS,HaU2,AiI}: For $T,U\in\tilt\Lambda$, $T\geq U$
if $\Ext^i_\Lambda(T,U)=0$ for all $i>0$. This is equivalent to that
$T^\bot\supseteq U^\bot$, where $T^\bot$ is the subcategory of
$\mod\Lambda$ consisting of modules $M$ such that
$\Ext_{\Lambda}^i(T,M)=0$ for any $i\geq1$. Clearly $\Lambda$ is the
maximal element in $\tilt\Lambda$, and if $\Lambda$ is
Iwanaga-Gorenstein, then $\mathbb{D}\Lambda$ is the minimal element
in $\tilt\Lambda$, where $\mathbb{D}$ is the ordinary duality.
However, it is difficult to find the minimal element in the set of
tilting $\Lambda$-modules for an arbitrary algebra $\Lambda$.

For a right $\Lambda$-module $M$, let $0\rightarrow M\rightarrow
I^0(M)\rightarrow I^1(M)\rightarrow\cdots$ be a minimal injective
resolution
 of $M$ and $\cdots\to P_{1}(M)\to P_{0}(M)\to M\to0$ a minimal projective resolution
 of $M$. Recall that an algebra $\Lambda$ is called {\it $n$-Gorenstein
(resp. quasi $n$-Gorenstein)} if  the projective dimension of
$I^i(\Lambda)$ is less than or equal to $i$ (resp. $i+1$) for $0\leq
i\leq n-1$ \cite{FGR,Hu}. There are many works on $n$-Gorenstein
algebras \cite{AuR3,AuR4,Cl,HuI,IwS}. But little is known for the
tilting modules over this class of algebras. Our first aim is to
study the existence of minimal tilting modules over this class of
algebras.
For a module $M\in\mod\Lambda$, we denote by $\Omega^i M$ (resp.
$\Omega^{-i}M$) the $i$-th syzygy (resp. cosyzygy) of $M$.
A special case of our first main Theorem \ref{2.9} is the following:

\begin{theorem}\label{1.1}(Corollary \ref{2.5})
Let $\Lambda$ be a quasi $n$-Gorenstein algebra and $0\le j\le n$.
Then
$(\bigoplus^{j-1}_{i=0}I^i(\Lambda))\oplus\Omega^{-j}\Lambda$ is the minimum element in $\tilt_j\Lambda$ .
\end{theorem}

For example, algebras with dominant dimension at least $n$ are
$n$-Gorenstein. In this case, the tilting module given in Theorem
\ref{1.1} was studied recently in \cite{CrS,NRTZ,PS} (see Example
\ref{2.7}).

Recall that a subcategory $\mathscr{C}$ of $\mod\Lambda$ is called
{\it contravariantly finite} if for any $M$ in $\mod\Lambda$ there
is a morphism $C_M\rightarrow M$ with $C_M\in\mathscr{C}$ such that
the functor sequence
$\Hom_{\Lambda}(-,C_M)\rightarrow\Hom_{\Lambda}(-,M)\rightarrow0$ is
exact over $\mathscr{C}$. Dually, one can define covariantly finite
subcategories. For a subcategory $\mathscr{D}$ of $\mod\Lambda$,
denote by $\mathscr{D}^{\bot}$ the subcategory consisting of modules
$N$ such that $\Ext_{\Lambda}^{i }(M,N)=0$ for $i\geq 1$ and $M\in
\mathscr{D}$. We denote by $M^{\bot}$ if $\mathscr{D}=\add M$.
Dually, one can define ${^{\bot}\mathscr{D}}$ and ${^{\bot}M}$.

Denote by $\PP_\infty(\Lambda)$ the subcategory consisting of
$\Lambda$-modules with finite projective dimension. Happel and Unger
in \cite[Theorem 3.3]{HaU2} showed that $\tilt\Lambda$ has a minimal
element if and only if $\PP_\infty(\Lambda)$ is contravariantly
finite. It is natural to ask if there is a similar result for
$\tilt_n\Lambda$. We give a positive answer by proving the following
result, where we denote by $\PP_n(\Lambda)$ the subcategory
consisting of modules with projective dimension at most $n$, where
the equivalence of (1) and (5) for $n=\infty$ recovers \cite[Theorem
3.3]{HaU2}, and the equivalence of (3) and (5) for integer $n$
recovers \cite[Corollary 2.3]{HaU1}.

\begin{theorem}\label{1.9}(Theorem \ref{2.12}) Let $\Lambda$ be an algebra, and let $n$ be $\infty$ or a non-negative
integer. Then the following are equivalent:
\begin{enumerate}[\rm(1)]
\item $\tilt_n\Lambda$ has a minimal element.

\item $\tilt_n\Lambda$ has the minimum element.

\item There exists $T\in\tilt_n\Lambda$ such that ${}^\bot T\supseteq\PP_n(\Lambda)$.

\item There exists $T\in\tilt_n\Lambda$ such that $^{\bot}({T^{\bot}})=\PP_n(\Lambda)$.

\item The subcategory $\PP_n(\Lambda)$ is contravariantly finite.
\end{enumerate}
\end{theorem}

For any $M\in\mod \Lambda$, denote by $\id_{\Lambda}M$ (resp.
$\pd_{\Lambda}M$) the injective (resp. projective) dimension of $M$.
An algebra is called {\it Iwanaga-Gorenstein} if both
$\id_{\Lambda}\Lambda$ and
 $\id_{\Lambda^{\rm op}}\Lambda$ are finite.
 In \cite{AuR3}, Auslander and Reiten posed a question which says
that if $\Lambda$ is $n$-Gorenstein for all positive integer $n$
then $\Lambda$ is Iwanaga-Gorenstein. This is a generalization of
the Nakayama conjecture which says that an algebra with infinite
dominant dimension is self-injective. Moreover, Auslander and Reiten
\cite[p25]{AuR3} studied the question whether the
$\PP_\infty(\Lambda)$ is contravariantly finite if $\Lambda$ is
$n$-Gorenstein for all positive integer $n$.

As a result of Theorem \ref{1.1} and Theorem \ref{1.9}, we connect
the two questions of Auslander and Reiten above and show the
following corollary which covers \cite[Corollary 5.5]{AuR3}.

\begin{corollary}(Corollary \ref{2.8}) Let $\Lambda$ be a
$k$-Gorenstein algebra for all positive integer $k$, and $n$ a
non-negative integer. Then the following are equivalent:
\begin{enumerate}[\rm(1)]
\item $\Lambda$ is Iwanaga-Gorenstein with $\id_{\Lambda}\Lambda=\id_{\Lambda^{\rm op}}\Lambda\le n$.

\item $\id_{\Lambda}\Lambda\le n$.

\item $\id_{\Lambda^{\rm op}}\Lambda\le n$.

\item $\tilt\Lambda$ has the minimum element $T$ with $\pd_\Lambda T\le n$.

\item $\tilt\Lambda^{\rm op}$ has the minimum element $T$ with $\pd_\Lambda T\le n$.

\item The subcategory $\PP_\infty(\Lambda)$ is contravariantly finite and $\PP_\infty(\Lambda)=\PP_n(\Lambda)$ holds.

\item The subcategory $\PP_\infty(\Lambda^{\rm op})$ is contravariantly finite and $\PP_\infty(\Lambda^{\rm op})=\PP_n(\Lambda^{\rm op})$ holds.
\end{enumerate}
\end{corollary}

It may be interesting to ask the following question for a finite
dimensional algebra $\Lambda$: Does the existence of the minimum
element of $\tilt\Lambda$ imply the existence of minimum element of
$\tilt\Lambda^{\rm op}$?

Now we turn to the classical tilting modules over $1$-Gorenstein
algebras and study the connections with $\tau$-tilting theory.

 In 2014, Adachi, Iyama and Reiten introduced $\tau$-tilting modules (see Definition \ref{3.1})
which are generalizations of classical tilting modules from the
viewpoint of mutation. For general details of $\tau$-tilting theory,
we refer to \cite{AIR,DIJ,IJY,J,W,Zh} and references.

For an algebra $\Lambda$, denote by $\sttilt\Lambda$ the set of the
isomorphism classes of basic support $\tau$-tilting
$\Lambda$-modules (see Definition \ref{3.1}). In \cite{DIRRT, IZ} it is showed that the functor
$-\otimes_{\Lambda}\Gamma$ induces a map from $\sttilt\Lambda$ to
$\sttilt\Gamma$, where $\Gamma$ is a factor algebra of $\Lambda$.
Recall that $\tilt_1\Lambda$ is the set of additive equivalence
classes of classical tilting $\Lambda$-modules. Our third main
result is the following.

\begin{theorem}\label{1.3}(Theorem \ref{3.5})
Let $\Lambda$ be an  $1$-Gorenstein algebra and let $\Gamma$ be the
factor algebra $\Lambda/(e)$, where $e$ is an idempotent such that
$e\Lambda$ is an additive generator of projective-injective modules.
Then $-\otimes_{\Lambda}\Gamma$ induces a bijection from
$\tilt_1\Lambda$ to $\sttilt\Gamma$.
\end{theorem}

 For an algebra $\Lambda$, denote by $\#\sttilt\Lambda$ the number of elements in the set $\sttilt\Lambda$. As an immediate consequence, we have the following corollary.
   Recall from \cite{DIJ} that an algebra
$\Lambda$ is called {\it $\tau$-tilting finite} if there are finite
number of basic $\tau$-tilting modules up to isomorphism.

\begin{corollary}\label{1.4}(Corollaries \ref{3.6}, \ref{4.5} and \ref{3.8})
For each case, let $e$ be the idempotent such that $e\Lambda_n$ is
the additive generator of projective-injective $\Lambda_n$-modules.
\begin{enumerate}[\rm(1)]
\setlength{\itemsep}{0pt}

 \item Let $\Lambda_n=KQ$ be the hereditary Nakayama algebra with $Q=A_n$.
 Then there are bijections
 \[\tilt_1{\Lambda_n}\simeq\sttilt\Lambda_{n-1}\simeq
\{\mbox{clusters of the cluster algebra of type $A_{n-1}$}\}.\]
Thus $\#\sttilt\Lambda_n=(2(n+1))!/((n+2)!(n+1)!)$.
\item Let $\Lambda_n$ be the Auslander algebra of $K[x]/(x^n)$, $\Gamma_{n-1}$
be the preprojective algebra of $Q=A_{n-1}$ and $\mathfrak{S}_{n}$ be the symmetric group.\emph{} Then there are bijections
\[\tilt_{1}\Lambda_{n}\simeq\sttilt\Gamma_{n-1}\simeq\mathfrak{S}_{n}.\]
Thus $\#\sttilt\Lambda_n=n!$.
\item Let $\Lambda_{n}$ be the Auslander algebra of the hereditary Nakayama
algebra $KQ$ with $Q=A_n$. Then there is a bijection
$\tilt_1{\Lambda_n}\simeq\sttilt\Lambda_{n-1}$.
Thus $\Lambda$ is $\tau$-tilting finite if and only if $n\leq 4$.
\end{enumerate}
\end{corollary}

Now we state the organization of this paper as follows:

In Section 2, we recall some preliminaries. In Section 3, we give
some equivalent conditions to the existence of minimal elements in
$\tilt_n\Lambda$ and show a Happel-Unger type theorem. Moreover,
 we construct minimal tilting modules for $n$-Gornestein algebras and show Theorem \ref{1.1}.
In Section 4, we build a connection between classical tilting
modules over 1-Gorenstein algebras and support $\tau$-tilting
modules over factor algebras and we show Theorem \ref{1.3} and
Corollary \ref{1.4}.

Throughout this paper, we denote by $K$ an algebraically closed
field. All algebras are basic connected finite dimensional
$K$-algebras and all modules are finitely generated right modules.
For an algebra $A$, we denote by
 $\mod A$ the category of finitely generated right $A$-modules.
The composition of homomorphisms $f:X \rightarrow Y$ and $g:Y
\rightarrow Z$ is denoted by $gf:X \rightarrow Z$.

\section{Preliminaries}

We start with the following fundamental theorem due to
Auslander-Reiten.

\begin{theorem}\cite[Theorem 5.5]{AuR1}\label{cotorsion pair}
Let $\Lambda$ be a finite dimensional algebra and $n\ge0$.
Then there exist bijections between the following objects given by $T\mapsto\XX={}^\bot(T^\bot)$ and $T\mapsto\YY=T^\bot$.
\begin{enumerate}[\rm(1)]
\item $T\in\tilt_n\Lambda$.
\item Contravariantly finite resolving subcategories $\XX$ of $\mod\Lambda$ contained in $\PP_n(\Lambda)$.
\item Covariantly finite coresolving subcategories $\YY$ of $\mod\Lambda$ containing $\Omega^{-n}(\Lambda)$.
\end{enumerate}
Moreover, in this case, $(\XX,\YY)$ is a cotorsion pair such that $\XX\cap\YY=\add T$.
\end{theorem}

In the rest, a \emph{subcategory} is always assumed to be full and closed under direct sums and direct summands.
For later application, we prepare the following observation, which is a relative version
of a well-known observation (e.g. \cite{AuS1}).

\begin{lemma}\label{two minimality}
Let $\CC$ be a subcategory of $\mod\Lambda$ which is closed under extensions,
and $\AA$ a subcategory of $\CC$.
Then the following conditions are equivalent.
\begin{enumerate}[\rm(1)]
\item Any exact sequence $0\to A\to A'\to C\to 0$ with $A,A'\in\AA$ and $C\in\CC$ splits.
\item There is no exact sequence $0\to A\to A'\to C\to 0$ with $A\in\ind\AA$, $A'\in\AA$,
$C\in\CC$ and $A\notin\add A'$.
\end{enumerate}
\end{lemma}

\begin{proof}
It suffices to prove (2)$\Rightarrow$(1).
Assume that there exists a non-split exact sequence
$0\to A\xrightarrow{f} A'\to C\to0$ with $A,A'\in\AA$ and $C\in\CC$.
Without loss of generality, we can assume that $f$ is in the radical of $\mod\Lambda$.

Take an indecomposable direct summand $X$ of $A$.
Let $\iota:X\to A$ be the inclusion and $g=f\iota$.
Then we have the following commutative diagram of non-split exact sequences.
\[\xymatrix@R1em{
0\ar[r]&X\ar[r]^g\ar[d]^\iota&A'\ar[r]\ar@{=}[d]&C'\ar[r]\ar[d]&0\\
0\ar[r]&A\ar[r]^f&A'\ar[r]&C\ar[r]&0\\
}\]
Since $0\to X\xrightarrow{\iota}A\to C'\to C\to0$ is an exact sequence, $C'$ belongs to $\CC$.

Decompose $A'=X^{\oplus\ell}\oplus U$ with $X\notin\add U$, and write
$g:X\to A'=X^{\oplus\ell}\oplus U$.
Let $Y^0=X$, $Y^i=X^{\oplus\ell^i}\oplus U^{\oplus(1+\ell+\cdots+\ell^{i-1})}$ and
\[g^i=g^{\oplus\ell^i}\oplus 1_U^{\oplus(1+\ell+\cdots+\ell^{i-1})}:Y^i\to Y^{i+1}.\]
Clearly $\Ker g^i=0$ and $\Coker g^i\in\CC$ holds for any $i\ge0$.

Take $m>0$ such that $\rad^m\End_\Lambda(X)=0$. Then the composition
\[h=g^{m-1}\cdots g^1g^0:X\to Y^m=X^{\oplus\ell^i}\oplus U^{\oplus(1+\ell+\cdots+\ell^{i-1})}\]
is a direct sum of morphisms $0\to X^{\oplus\ell^m}$ and
$h':X\to U^{\oplus(1+\ell+\cdots+\ell^{i-1})}$. Thus we have an exact sequence
\begin{equation}\label{X unnecessary}
0\to X\to U^{\oplus(1+\ell+\cdots+\ell^{i-1})}\to\Coker h'\to0.
\end{equation}
Since $\Coker h$ is an extension of $\Coker g^i$'s, it belongs to $\CC$.
Thus $\Coker h'$ also belongs to $\CC$. This is contradiction to the condition (2).
\end{proof}
Let $\CC$ be a subcategory of $\mod\Lambda$ which is closed under extensions.
A \emph{cogenerator} for $\CC$ is a subcategory $\AA$ of $\CC$ such that,
for any $C\in\CC$, there exists an exact sequence $0\to C\to A\to C'\to0$
with $A\in\AA$ and $C'\in\CC$.
A cogenerator $\AA$ of $\CC$ is called \emph{minimal} if no proper subcategory of $\AA$
is a cogenerator of $\CC$.

The following observation is a relative version of a well-known result in
Auslander-Smalo's theory on (co)covers \cite{AuS1}.

\begin{proposition}\label{cover}
Let $\CC$ be a subcategory of $\mod\Lambda$ which is closed under extensions.
If $\AA$ is a minimal cogenerator for $\CC$, then $\Ext^1_\Lambda(\CC,\AA)=0$ holds.
\end{proposition}

\begin{proof}
Since $\AA$ is a minimal cogenerator for $\CC$, the conditions (1) and (2) in Lemma \ref{two minimality} are satisfied. Otherwise, there is an exact sequence $0\to A\to A'\to C\to 0$ with $A\in\ind\AA$, $A'\in\AA$, $C\in\CC$ and $A\notin\add A'$.
Then the subcategory $\AA'$ of $\AA$ defined by $\ind\AA'=(\ind\AA)\setminus\{A\}$ is a cogenerator for $\CC$, a contradiction to the minimality of $\AA$.

Let $A\in\AA$ be indecomposable. To prove $\Ext^1_\Lambda(\CC,A)=0$,
we take an exact sequence $0\to A\to C'\to C\to0$ with $C\in \CC$ and $A\in\AA$.
Since $C'\in \CC$, there exists an exact sequence $0\to C'\to A'\to C''\to0$
with $C''\in\CC$ and $A'\in\AA$. We have the following commutative
diagram of exact sequences
\[\xymatrix@R1em{
&&0\ar[d]&0\ar[d]\\
0\ar[r]&A\ar[r]\ar@{=}[d]&C'\ar[r]\ar[d]&C\ar[r]\ar[d]&0\\
0\ar[r]&A\ar[r]&A'\ar[r]\ar[d]&X\ar[r]\ar[d]&0\\
&&C''\ar[d]\ar@{=}[r]&C''\ar[d]\\
&&0&0
}\]
By the right vertical sequence, $X$ belongs to $\CC$.
Thus the middle horizontal sequence splits by Lemma \ref{two minimality}(1).
Therefore the upper horizontal sequence splits, as desired.
\end{proof}

We need the following result on mutation of tilting modules.

\begin{proposition}\cite{HaU2,CHU,AiI}\label{tilting mutation}
For a basic tilting module $T=X\oplus U$ with $X$ indecomposable, if
there is an exact sequence $0\to X\stackrel{f}{\to}U'\to Y\to0$ such that $U'\in\add U$
and $f$ is a minimal left $\add U$-approximation of $X$, then $V=Y\oplus U$ is a basic tilting module such that $V<T$.
\end{proposition}

\section{Minimal tilting modules and the category $\PP_n(\Lambda)$}

\subsection{Characterizations of existence of minimal tilting modules}

Throughout this section, let $\Lambda$ be an arbitrary algebra. We focus on the properties of tilting modules in
$\tilt_n\Lambda$ and give some equivalent conditions to the
existence of minimal elements in $\tilt_n\Lambda$. More precisely,
we generalize Happel-Unger theorem stating that $\tilt\Lambda$ has a
minimal element if and only if $\PP_\infty(\Lambda)$ is
contravariantly finite (see \cite[Theorem 3.3]{HaU2}).
 Now we connect the existence of a minimal element in $\tilt_n\Lambda$ with
the contravariantly finiteness of $\PP_n(\Lambda)$ and show our
main result below. Note that the equivalence of (1) and (5)
for $n=\infty$ recovers \cite[Theorem 3.3]{HaU2}, and the
equivalence of (3) and (5) for integer $n$ recovers \cite[Corollary
2.3]{HaU1}.

\begin{theorem}\label{2.12} Let $\Lambda$ be an algebra, and let $n$ be $\infty$ or a non-negative
integer. Then the following are equivalent:
\begin{enumerate}[\rm(1)]
\item $\tilt_n\Lambda$ has a minimal element.

\item $\tilt_n\Lambda$ has the minimum element.

\item There exists $T\in\tilt_n\Lambda$ such that ${}^\bot T\supseteq\PP_n(\Lambda)$.

\item There exists $T\in\tilt_n\Lambda$ such that $^{\bot}({T^{\bot}})=\PP_n(\Lambda)$.

\item The subcategory $\PP_n(\Lambda)$ is contravariantly finite.
\end{enumerate}
\end{theorem}

To prove Theorem \ref{2.12}, we need the following result.

\begin{proposition}\label{characterizations of minimal}
Let $\Lambda$ be an algebra, and let $\CC$ be a resolving
subcategory of  $\mod\Lambda$ contained in $\PP_{\infty}(\Lambda)$.
For $T\in\CC\cap\tilt\Lambda$, the following conditions are
equivalent.
\begin{enumerate}[\rm(1)]
\item $T$ is a minimal element in $\CC\cap\tilt\Lambda$.

\item $T$ is the minimum element in $\CC\cap\tilt\Lambda$.

\item $^{\bot}({T^{\bot}})=\CC$.

\item ${}^\bot T\supseteq\CC$.
\item $\CC\cap T^{\bot}=\add T$.

\item Every exact sequence $0\to T_1\to T_0\to C\to0$ with $T_i\in\add T$ and
$C\in\CC$ splits.

\item There is no monomorphism $f:X\to T'$ such that $X$ is an indecomposable
direct summand of $T$, $T'\in\add(T/X)$ and $\Coker f\in\CC$.
\end{enumerate}
\end{proposition}

\begin{proof}
Note that, for any $U\in\CC\cap\tilt\Lambda$, we have
\begin{equation}\label{XU,YU}
U^\bot\supseteq\CC^\bot\ \mbox{ and }\ {}^\bot(U^\bot)\subseteq\CC.
\end{equation}

(1)$\Rightarrow$(7) Assume that there exists such $f:X\to T'$.
Let $g:X\to U'$ be a minimal left $(\add T/X)$-approximation and $Y=\Coker g$.
Then $f$ factors throught $g$, and we have a commutative diagram of exact sequences
\[\xymatrix@R1em{
0\ar[r]&X\ar[r]^f\ar@{=}[d]&T'\ar[r]&\Coker f\ar[r]&0\\
&X\ar[r]^g&U'\ar[r]\ar[u]&Y\ar[r]\ar[u]&0 }\] Thus $\Ker g=0$, and
we have an exact sequence $0\to U'\to T'\oplus Y\to\Coker f\to0$.
Thus $Y\in\CC$ holds. This means that $U=(T/X)\oplus Y$ is a
mutation of $T$ (Proposition \ref{tilting mutation}) and gives an
element of $\CC\cap\tilt\Lambda$ such that $T>U$, a contradiction.

(7)$\Rightarrow$(6) We only need to apply Lemma \ref{two minimality}
for $\AA=\add T$.


(6)$\Rightarrow$(5) It suffices to show $\CC\cap
T^{\bot}\subseteq\add T$. Let $C_0\in\CC\cap T^{\bot}$ and
$n=\pd_\Lambda C_0$. Since $T^\bot$ is an exact category with enough
projectives $\add T$, there exists an exact sequence
\[\cdots\xrightarrow{f_2} T_1\xrightarrow{f_1} T_0\xrightarrow{f_0} C_0\to0\]
such that $T_i\in\add T$ and $C_i=\Im f_i$ belongs to $T^{\bot}$.
Since $\CC$ is resolving, each $C_i$ belongs to $\CC$. For
every $i>0$, we have
\[\Ext^{i}_\Lambda(C_n,T^\perp)\simeq\Ext^{i+1}_\Lambda(C_{n-1},T^\perp)\simeq
\cdots\simeq\Ext^{i+n}_\Lambda(C_0,T^\perp)=0.\]

Thus $C_n$ belongs to $T^\bot\cap{}^\bot(T^\bot)=\add T$ by Theorem
\ref{cotorsion pair}.

Since $C_{n-1}\in\CC$, the exact sequence $0\to C_n\to T_{n-1}\to
C_{n-1}\to0$ splits by our assumption, and hence $C_{n-1}$ belongs
to $\add T$. Repeating the same argument, we have $C_0\in\add T$, as
desired.

(5)$\Rightarrow$(1) Assume $U\in\CC\cap\tilt\Lambda$ satisfies $T\ge
U$. Then $U\in\CC\cap T^\bot=\add T$ and hence $U=T$.

(5)+(7)$\Rightarrow$(4) By Proposition \ref{cover}, it suffices to
show that $\add T$ is a minimal cogenerator for $\CC$. Let
$C\in\CC$. Since $({}^\bot(T^\bot),T^\bot)$ is a cotorsion pair by
Theorem \ref{cotorsion pair}, there exists an exact sequence $0\to
C\to Y\to X\to0$ with $Y\in T^\bot$ and $X\in{}^\bot(T^\bot)$. By
\eqref{XU,YU}, we have $X\in\CC$. Since $\CC$ is extension closed,
we have $Y\in\CC\cap T^\bot=\add T$ by (5). Thus $\add T$ is a
cogenerator for $\CC$. Moreover, it is minimal by (7).

(4)$\Rightarrow$(3) Thanks to \eqref{XU,YU}, it suffices to show
${}^\bot(T^\bot)\supseteq\CC$, i.e.\ any $X\in\CC$ and $Y_0\in
T^\bot$ satisfy $\Ext^{i}_\Lambda(X,Y_0)=0$ for all $i>0$. Since
$T^\perp$ is an exact category with enough projectives $\add T$,
there exists an exact sequence
\[\cdots\xrightarrow{f_2} T_1\xrightarrow{f_1} T_0\xrightarrow{f_0} Y_0\to0\]
such that $T_i\in\add T$ and $Y_i=\Im f_i$ belongs to $T^{\bot}$.
For $n=\pd_\Lambda X$, by using (4), we have
\[\Ext^{i}_\Lambda(X,Y_0)\simeq\Ext^{i+1}_\Lambda(X,Y_1)\simeq
\cdots\simeq\Ext^{i+n}_\Lambda(X,Y_{n})=0\] as desired.

(3)$\Rightarrow$(2) Let $U\in\CC\cap\tilt\Lambda$. By \eqref{XU,YU},
we have ${}^\bot(U^\bot)\subseteq\CC={}^\bot(T^\bot)$. Thus
$U^\bot\supseteq T^\bot$ by Theorem \ref{cotorsion pair}.

(2)$\Rightarrow$(1) Clear.
\end{proof}

Now we prove the following theorem, where the equivalence of (3) and (5) recovers \cite[Theorems
2.1, 2.2]{HaU1}.

\begin{theorem}\label{general 2.12}
Let $\Lambda$ be an algebra, and let $\CC$ be a resolving
subcategory of  $\mod\Lambda$ contained in $\PP_{\infty}(\Lambda)$.
Then the following are equivalent:
\begin{enumerate}[\rm(1)]
\item $\CC\cap\tilt\Lambda$ has a minimal element.

\item $\CC\cap\tilt\Lambda$ has the minimum element.

\item There exists $T\in\CC\cap\tilt\Lambda$ such that ${}^\bot T\supseteq\CC$.

\item There exists $T\in\CC\cap\tilt\Lambda$ such that $^{\bot}({T^{\bot}})=\CC$.

\item The subcategory $\CC$ is contravariantly finite.
\end{enumerate}
\end{theorem}

\begin{proof}
(1)$\Leftrightarrow$(2)$\Leftrightarrow$(3)$\Leftrightarrow$(4)
These are shown in Proposition \ref{characterizations of minimal}.

(4)$\Leftrightarrow$(5) This is well-known (see Theorem \ref{cotorsion
pair}).
\end{proof}

We are ready to prove Theorem \ref{2.12}.

\begin{proof}[Proof of Theorem \ref{2.12}]
One can get the assertion by putting $\CC=\PP_n(\Lambda)$ in Theorem \ref{general 2.12}.
\end{proof}

\subsection{Minimal tilting modules of Auslander-Gorenstein algebras}

In this section, we construct a class of minimal tilting modules for
$n$-Gorenstein algebras. As a result, we show some equivalent
conditions for an $n$-Gorenstein algebra to be Iwanaga-Gorenstein,
which gives a partial answer to a question of Auslander and Reiten
mentioned before.

 Now we have the following result which gives a method
in constructing minimal tilting modules with finite projective
dimension.

\begin{theorem}\label{2.9} For an algebra $\Lambda$ and a fixed integer $n\geq0$, assume that $\pd_{\Lambda} I^i(\Lambda)\leq
n$ for any $i$, $0\leq i\leq n-1$ and
$\pd_{\Lambda}\Omega^{-n}\Lambda\leq n$. Let
$T=(\bigoplus_{i=0}^{n-1}I^i(\Lambda))\oplus\Omega^{-n}\Lambda$.
Then we have the following:
\begin{enumerate}[\rm(1)]
\item $T$ is a tilting module with projective dimension at most $n$.
\item $T$ is the minimum element in $\tilt_{n}\Lambda$.
\item $\PP_n(\Lambda)$ is contravariantly finite and $^{\bot}({T^{\bot}})=\PP_n(\Lambda)$.
\end{enumerate}
\end{theorem}

\begin{proof}
We show the assertion (1) step by step.

$\bullet$ By our assumptions, $\pd_{\Lambda}T\leq n$ holds.

$\bullet$ $\Ext_{\Lambda}^{i}(T,T)=0$ for $i\geq 1$. It suffices to
show $\Ext_{\Lambda}^{i}(\Omega^{-n}\Lambda,\Omega^{-n}\Lambda)=0$
for $i\geq 1$, and
$\Ext_{\Lambda}^{i}(I^j(\Lambda),\Omega^{-n}\Lambda)=0$ for $i\geq
1$ and $1\leq j\leq n-1$.

Applying the functor $\Hom_{\Lambda}(\Omega^{-n}\Lambda,-)$ to the
following exact sequence
\begin{equation}\label{1}
\xymatrix{0\ar[r]&\Lambda\ar[r] &I^0(\Lambda)\ar[r] &\cdots\ar[r]
&I^{n-1}(\Lambda)\ar[r] &\Omega^{-n}\Lambda\ar[r]&0}
\end{equation}
one can show that
$\Ext_{\Lambda}^{i}(\Omega^{-n}\Lambda,\Omega^{-n}\Lambda)\simeq
\Ext_{\Lambda}^{i+n}(\Omega^{-n}\Lambda,\Lambda)=0$ holds for $i\geq
1$ since $\pd_{\Lambda}\Omega^{-n}\Lambda\leq n$. Applying the
functor $\Hom_{\Lambda}(I^j(\Lambda),-)$ to the exact sequence
\eqref{1}, we get that
$\Ext_{\Lambda}^{i}(I^j(\Lambda),\Omega^{-n}\Lambda)\simeq
\Ext_{\Lambda}^{i+n}(I^j(\Lambda),\Lambda)=0$ holds for $i\geq1$
since $\pd_{\Lambda}I^j(\Lambda)\leq j\leq n-1$.

 $\bullet$ The exact sequence \eqref{1} is our desired resolution
in the definition of tilting modules.

(2) It suffices to show that  $T\in U^{\bot}$ holds for any tilting
module $U$ with $\pd_{\Lambda}U\leq n$. Then by the definition of
$T$, we only have to show $\Omega^{-n}\Lambda\in U^{\bot}$. Applying
the functor $\Hom_{\Lambda}(U,-)$ to the exact sequence \eqref{1},
 one can get that
$\Ext_{\Lambda}^i(U,\Omega^{-n}\Lambda)\simeq\Ext_{\Lambda}^{n+i}(U,\Lambda)=0$
holds for $i\geq1$ since $\pd_{\Lambda}U\leq n$. Then the assertion
follows.

(3) This follows from Theorem \ref{2.12}.
\end{proof}

Immediately, we have the following corollary.

\begin{corollary}\label{2.5} Let $\Lambda$ be a quasi $n$-Gorenstein algebra with $n\geq 0$.
Then $\tilt_n\Lambda$ has the minimum element
$(\bigoplus_{i=0}^{n-1}I^i(\Lambda))\oplus\Omega^{-n}\Lambda$.
\end{corollary}

\begin{proof}
Immediate from Theorem \ref{2.9}
\end{proof}

Recall that an algebra $\Lambda$ is called of {\it dominant
dimension $n$} if $I^i(\Lambda)$ is projective for $0\leq i\leq n-1$
but $I^n(\Lambda)$ is not projective. Then we have the following
immediate from Theorem \ref{2.9}.

\begin{example}\label{2.7} Let $\Lambda$ be an algebra with dominant
dimension $n\geq 0$ and let
$T=I^0(\Lambda)\oplus\Omega^{-n}\Lambda$. Then $T$ is the minimum
element in $\tilt_{n}\Lambda$, which was studied recently in
\cite{CrS,Ma,NRTZ,PS}. The equality
$^{\bot}({T^{\bot}})=\PP_n(\Lambda)$ in Theorem \ref{2.9}(3) was
observed in \cite[2.5]{Ma}.
\end{example}

Now we give some applications to a question of
Auslander and Reiten which says that if $\Lambda$ is $n$-Gorenstein
for all non-negative integer $n$ then $\Lambda$ is
Iwanaga-Gorenstein. This is a generalization of the famous Nakayama
conjecture. We have the following:

\begin{corollary}\label{2.8} Let $\Lambda$ be a
$k$-Gorenstein algebra for all positive integer $k$, and $n$ a
non-negative integer. Then the following are equivalent:
\begin{enumerate}[\rm(1)]
\item $\Lambda$ is Iwanaga-Gorenstein with $\id_{\Lambda}\Lambda=\id_{\Lambda^{\rm op}}\Lambda\le n$.

\item $\id_{\Lambda}\Lambda\le n$.

\item $\id_{\Lambda^{\rm op}}\Lambda\le n$.

\item $\tilt\Lambda$ has the minimum element $T$ with $\pd_\Lambda T\le n$.

\item $\tilt\Lambda^{\rm op}$ has the minimum element $T$ with $\pd_{\Lambda^{\rm op}} T\le n$.

\item The subcategory $\PP_\infty(\Lambda)$ is contravariantly finite, and $\PP_\infty(\Lambda)=\PP_n(\Lambda)$.

\item The subcategory $\PP_\infty(\Lambda^{\rm op})$ is
contravariantly finite, and $\PP_\infty(\Lambda^{\rm
op})=\PP_n(\Lambda^{\rm op})$.
\end{enumerate}
\end{corollary}

\begin{proof}
$(1)\Rightarrow (2)$ is clear, $(2)\Leftrightarrow (3)$ follows from
\cite[Corollary 5.5]{AuR3}. Hence
$(1)\Leftrightarrow(2)\Leftrightarrow (3)$ holds.
$(2)\Leftrightarrow(4)$ follows from Corollary \ref{2.5}.
$(4)\Leftrightarrow (6)$ follows from \cite[Theorem 3.3]{HaU2} (see
Theorem \ref{2.12}). Dually, $(3)\Leftrightarrow(5)\Leftrightarrow
(7)$ holds.
\end{proof}

We give an example to show the existence and the constructing of minimal tilting
modules.

\begin{example}\label{4.1} Let $\Lambda=KQ/I$ be an algebra with the quiver $Q$:
\[\xymatrix{1\ar[r]^{a_1}
&2\ar[r]^{a_2} &\cdots\ar[r]^{a_{n-1}}& n\ar[r]^{a_n} &n+1}\] and
$I={\rm Rad}^2KQ$. Then
\begin{enumerate}[\rm(1)]
\item The global dimension of $\Lambda$ is $n$ and $P(i)=I(i+1)$ for
$1\leq i\leq n$
\item The minimal injective resolution of $\Lambda$ is as follows:

$0\rightarrow \Lambda\rightarrow(\bigoplus_{i=2}^{n+1}I(i)) \oplus
I(n+1)\rightarrow I(n)\rightarrow \cdots\rightarrow I(1)\rightarrow
0$.

Hence $\Lambda$ is $n$-Gorenstein.
\item The tilting module
$T_j=(\bigoplus_{i=2}^{n+1}I(i))\oplus S(n-j+1)$ is of projective
dimension $j$ for $0\leq j\leq n$ and $T_j$ is a minimal element in
the set $\tilt_{j}\Lambda$.
\end{enumerate}
\end{example}

\subsection{The category $\PP_n(\Lambda)^\bot$}
Since $\PP_n(\Lambda)$ is a resolving subcategory of $\mod\Lambda$,
it is very natural to study the corresponding coresolving subcategory
$\PP_n(\Lambda)^\bot$. Clearly the subcategory
\[\YY_n^1(\Lambda)=\add\Omega^{-n}(\mod\Lambda)\]
is contained in $\PP_n(\Lambda)^\bot$ and satisfies $^\bot\YY_n^1(\Lambda)=\PP_n(\Lambda)$.
Auslander and Reiten \cite[Theorem 1.2]{AuR3} proved that
$\YY_n^1(\Lambda)$ is always covariantly finite. Moreover they proved that
$\YY_i^1(\Lambda)$ is extension closed (or equivalently, coresolving)
for every $1\le i\le n$ if and only if $\Lambda$ is quasi $n$-Gorenstein \cite[Theorem 2.1]{AuR3}.
For more general class of algebras, it is natural to consider the extension closure
of $\YY_n^1(\Lambda)$.

For a subcategory $\mathscr{C}$ of $\mod\Lambda$, denote by
$\mathscr{C}^{*m}$ consisting of all $X\in\mod\Lambda$ having a
filtration $X=X_0\supseteq X_1\supseteq\cdots\supseteq X_m =0$ such that
$X_i/ X_{i+1}\in\mathscr{C}$. Denote by $\Omega^{-n}(\mod\Lambda)$
consisting of the modules of the form $\Omega^{-n}C\oplus I$ for all
$C$ in $\mod\Lambda$ and injective modules $I$. Denote by

\[\YY_n^m(\Lambda)=\add(\Omega^{-n}(\mod\Lambda)^{*m})\ \mbox{ and }\
\YY_n(\Lambda)=\bigcup_{m\geq0}\YY_n^m(\Lambda).\]

We have the following observations.

\begin{proposition}\label{Ynm}
Let $\Lambda$ be an algebra and $n$ a non-negative integer.
\begin{enumerate}[\rm(1)]
\item $\YY_n^m(\Lambda)$ is a covariantly finite subcategory for every $m>0$ such that $^\bot\YY_n^m(\Lambda)=\PP_n(\Lambda)$.
\item $\YY_n(\Lambda)$ is a coresolving subcategory such that $^\bot\YY_n(\Lambda)=\PP_n(\Lambda)$.
\end{enumerate}
\end{proposition}

\begin{proof}
Clearly $^\bot\YY_n(\Lambda)$ and ${}^\bot\YY_n^m(\Lambda)$ coincide with ${}^\bot\YY_n^1(\Lambda)=\PP_n(\Lambda)$.
For (1), we refer to \cite{C}. The assertion (2) is clear.
\end{proof}

As a consequence, we have the following observations.

\begin{proposition}\label{2.13} For an algebra $\Lambda$,
consider the following five conditions:
\begin{enumerate}[\rm(1)]
\item The subcategory $\PP_n(\Lambda)$ is
contravariantly finite.

\item $\YY_n(\Lambda)$ is covariantly finite.

\item $\YY_n^m(\Lambda)$ is closed under extensions for some $m>0$.

\item $\YY_n(\Lambda)=\YY_n^m(\Lambda)$ holds for
some $m>0$

\item $\YY_n^1(\Lambda)$ is closed under extensions.

\end{enumerate}
Then we have
$(5)\Rightarrow(4)\Leftrightarrow(3)\Leftrightarrow(2)\Rightarrow(1)$.
\end{proposition}

\begin{proof}
$(5)\Rightarrow (3)$ We may choose $m=1$.

$(3)\Rightarrow(2)$ Assume that $\YY_n^m(\Lambda)$ is extension closed.
Then $\YY_n(\Lambda)=\YY_n^m(\Lambda)$.
This is covariantly finite by Proposition \ref{Ynm}(1).


$(2)\Rightarrow (4)$ Let $\ell$ is the Loewy length of $\Lambda$,
$S=\Lambda/\rad\Lambda$ and $S\to Y$ a left
$\YY_n(\Lambda)$-approximation of $S$.
Then $Y$ belongs to $\YY_n^m(\Lambda)$ for some $m>0$.
We claim $\YY_n(\Lambda)=\YY_n^{m\ell}(\Lambda)$.

Since any $X\in\mod\Lambda$ belongs to $(\add S)^{*\ell}$,
the Horseshoe-type Lemma \cite[Proposition 3.6]{AuR1} shows that
$X$ has a left $\YY_n(\Lambda)$-approximation $X\rightarrow Y$
with $Y\in(\add Y)^{*\ell}\subseteq\YY_n^{m\ell}(\Lambda)$.
If $X\in\YY_n(\Lambda)$, then $f$ is a split monomorphism.
Thus $X\in\YY_n^{m\ell}(\Lambda)$ holds.

$(4)\Rightarrow(3)$ This is clear since $\YY_n(\Lambda)$ is extension closed.

$(2)\Rightarrow (1)$ By Proposition \ref{Ynm}(2), $\YY_n(\Lambda)$ is a covariantly
finite coresolving subcategory such that $^\bot\YY_n(\Lambda)=\PP_n(\Lambda)$.
By \cite[Lemma 3.3 (a)]{AuR2}, $\PP_n(\Lambda)$ is contravariantly finite.
\end{proof}

We should remark that $(5)$ is not equivalent to $(2)$ in
Proposition \ref{2.13}. We give an example to show this.
However, we do not know whether $(2)$ is equivalent to $(1)$ or not.

\begin{example}\label{4.4}
Let $\Lambda$ be a local algebra with Loewy length $2$. Then
$\YY_1^1(\Lambda)=\add\{K,\mathbb{D}\Lambda\}$ holds.
Thus it is closed under extensions if and only if $\Lambda$ is self-injective.

On the other hand, $\YY_1^2(\Lambda)=\mod\Lambda$ holds,
and hence Proposition \ref{2.13}(3) is satisfied.
Moreover, $\tilt_1\Lambda$ has a minimal element $\Lambda$,
and $\PP_1(\Lambda)=\add\Lambda$ is contravariantly finite.
\end{example}

\section {A bijection between classical tilting modules and support $\tau$-tilting modules}

 Throughout this section,
$\Lambda$ is an  $1$-Gorenstein algebra and $e$ is an idempotent
such that $e\Lambda$ is an additive generator of
projective-injective modules. Denote by $\Gamma=\Lambda/(e)$ the
factor algebra of $\Lambda$. We mainly focus on the bijection between
classical tilting modules over an $1$-Gorenstein algebra $\Lambda$
and support $\tau$-tilting modules over the factor algebra $\Gamma$.

Denote by $\tau$ the AR-translation and denote by $|N|$ the number
of non-isomorphic indecomposable direct summands of $N$ for a
$\Lambda$-module $N$. Firstly, we recall the definition of support
$\tau$-tilting modules in \cite{AIR}.

\begin{definition}\label{3.1}
\begin{enumerate}[\rm(1)]
\item We call $N \in \mod \Lambda$ {\it $\tau$-rigid} if ${\rm
Hom}_{\Lambda}(N, \tau N) = 0$.
\item We call $N \in \mod \Lambda$ {\it $\tau$-tilting} if $N$ is
$\tau$-rigid and $|N| = |\Lambda|$.
\item We call $N \in \mod \Lambda$ {\it support $\tau$-tilting} if there
exists an idempotent $e$ of $\Lambda$ such that $N$ is a
$\tau$-tilting $(\Lambda/(e))$-module.
\end{enumerate}
\end{definition}

The following property is also needed for the main result in this
section.

\begin{lemma}\label{3.2}\cite{AIR} For an algebra $\Lambda$, classical
tilting $\Lambda$-modules are precisely faithful support
$\tau$-tilting $\Lambda$-modules.
\end{lemma}

Now we are in a position to state the following properties of
tilting modules over  $1$-Gorenstein algebras.

\begin{lemma}\label{3.3} For an $1$-Gorenstein algebra $\Lambda$
and a classical tilting $\Lambda$-module $T$, then
\begin{enumerate}[\rm(1)]
\item Each indecomposable projective-injective $\Lambda$-module is a direct summand of $T$.
\item Every support $\tau$-tilting $\Lambda$-module $M$ admitting $e\Lambda$
as a direct summand is a classical tilting module.
\end{enumerate}
\end{lemma}

\begin{proof} (1) Since $T$ is a classical tilting module, by Lemma \ref{3.2}
$T$ is faithful, and hence any injective module $I$ is generated by
$T$. Then we get $I$ is a direct summand of $T$ since $I$ is
projective and $T$ is basic.

(2) Since $\Lambda$ is $1$-Gorenstein, then $\Lambda$ can be
embedded in $\add_{\Lambda}e\Lambda$, and hence $\Lambda$ can be
embedded in $\add_{\Lambda}M$. Then $M$ is faithful, by Lemma
\ref{3.2}, $M$ is a classical tilting module.
\end{proof}

For an algebra $A$, denote by $\sttilt_{U}A$ the set of support
$\tau$-tilting $A$-modules with a direct summand $U$. Denote by
$U^{\bot_0}$ the subcategory consisting of modules $M$ such that
$\Hom_{\Lambda}(U,M)=0$. The following theorem \cite{J} is essential
to the main result in this section.

\begin{theorem}\label{3.4}
Let $A$ be an algebra and let $U$ be a basic $\tau$-rigid
$A$-module. Let $T$ be the Bongartz completion of $U$, $B=\End_{A}T$
and $C=B/(e_U)$, where $e_U$ is the idempotent corresponding to the
projective $B$-module $\Hom_{A}(T,U)$. Then there is a bijection
$\phi: \sttilt_{U}A\rightarrow\sttilt C$ via $M\rightarrow
\Hom_{A}(T,fM)$, where $0\rightarrow tM\rightarrow M\rightarrow
fM\rightarrow0$ is the canonical sequence according to the torsion
pair $(\Fac U,U^{\bot_0})$.
\end{theorem}

Recall that $\tilt_1 \Lambda$ is the set of additive equaivalence
classes of classical tilting $\Lambda$-modules. Now we are in a
position to show our main result in this section.

\begin{theorem}\label{3.5}
Let $\Lambda$ be an $1$-Gorenstein algebra and $\Gamma=\Lambda/(e)$,
where $e$ is an idempotent such that $e\Lambda$ is a basic part of
projective-injective modules. Then the tensor functor
$-\otimes_{\Lambda}\Gamma$ induces a bijection from $\tilt_1
\Lambda$ to $\sttilt \Gamma$.
\end{theorem}
\begin{proof} Since $\Lambda$ is  $1$-Gorenstein, we get
a basic additive generator $e\Lambda$ of projective-injective
modules. Let $U=e\Lambda$, then every support $\tau$-tilting
$\Lambda$-module $N$ admitting $U$ as a direct summand is faithful
since $U=e\Lambda$ is faithful. Then by Lemma \ref{3.2} or Lemma
\ref{3.3}, $N$ is a classical tilting $\Lambda$-module. On the other
hand, the Bongartz completion of $U$ is nothing but $T=\Lambda$.
Then $\End_{\Lambda}T=\Lambda$ and $\Hom_{\Lambda}(T,U)\simeq
\Hom_{\Lambda}(\Lambda, U)\simeq U$ as right $\Lambda$-modules. Now
we get $e=e_U$ in Theorem \ref{3.4}. Then we get a bijection $\phi$
from $\tilt_1\Lambda$ to $\sttilt\Gamma$.

In the following we show $-\otimes_{\Lambda}\Gamma=\phi$ as a map.
Note that $U=e\Lambda$, then the canonical sequence of a tilting
module $T'$ according to the torsion pair $(\Fac U, U^{\bot_0})$ is
$0\rightarrow T'(e)\rightarrow T'\rightarrow T'/T'(e)\rightarrow 0$.
Then by Theorem 4.4, $\phi(T')=\Hom_{\Lambda}(\Lambda,
T'/T'(e))\simeq T'/T'(e)$ as a support $\tau$-tilting
$\Gamma$-module. By \cite{DIRRT}, $-\otimes_{\Lambda}\Gamma$ is a
map from $\tilt_1 \Lambda$ to $\sttilt\Gamma$. Note that
$T'\otimes_{\Lambda}\Gamma\simeq T'/T'(e)$ for any $T'\in
\tilt_1\Lambda$, then $\phi=-\otimes_{\Lambda}\Gamma$ as a map.
\end{proof}

To give some applications of Theorem \ref{3.5}, we prepare the following facts.

\begin{example}\label{4.2} Let $\Lambda_n$ be the Auslander algebra of $K[x]/(x^n)$ and let $\Gamma_{n-1}$ be the preprojective
algebra of type $A_{n-1}$. Then we have the following:
\begin{enumerate}[\rm(1)]
\item  $\Lambda_n$ is given by the quiver: \[\xymatrix{
1\ar@<2pt>[r]^{a_1}&2\ar@<2pt>[r]^{a_2}\ar@<2pt>[l]^{b_2}&3\ar@<2pt>[r]^{a_3}\ar@<2pt>[l]^{b_3}&\cdots\ar@<2pt>[r]^{a_{n-2}}\ar@<2pt>[l]^{b_4}&n-1\ar@<2pt>[r]^{a_{n-1}}\ar@<2pt>[l]^{b_{n-1}}&n\ar@<2pt>[l]^{b_n}
}\] with relations $a_{1} b_{2}= 0$ and $a_{i} b_{i+1} =b_{i}
a_{i-1}$ for any $2 \leq i \leq n-1$.

\item $\Gamma_{n-1}$ is given by the quiver: \[\xymatrix{
1\ar@<2pt>[r]^{a_1}&2\ar@<2pt>[r]^{a_2}\ar@<2pt>[l]^{b_2}&3\ar@<2pt>[r]^{a_3}\ar@<2pt>[l]^{b_3}&\cdots\ar@<2pt>[r]^{a_{n-2}}\ar@<2pt>[l]^{b_4}&n-1\ar@<2pt>[l]^{b_{n-1}}
}\] with relations $a_{1} b_{2}= 0$, $a_{n-2}b_{n-1}=0$ and $a_{i}
b_{i+1} =b_{i} a_{i-1}$ for any $2 \leq i \leq n-2$.

\end{enumerate}
\end{example}

In the rest of this section, for an algebra $\Lambda_n$, we always assume that $e$ is the idempotent such that $e\Lambda_n$ is the additive generator of projective-injective $\Lambda_n$-modules.
Applying Theorem \ref{3.5} to the Auslander algebras of
$K[x]/(x^n)$, we get the following corollary which recovers the
results in \cite{Mi,IZ}.

\begin{corollary}\label{3.6}
Let $\Lambda_n$ be the Auslander algebra of $K[x]/(x^n)$,
$\Gamma_{n-1}$ be the preprojective algebra of $Q=A_{n-1}$ and $\mathfrak{S}_{n}$ be the symmetric group.  Then
there are bijections
\[\tilt_{1}\Lambda_{n}\simeq\sttilt\Gamma_{n-1}\simeq\mathfrak{S}_{n}.\]
Thus $\#\sttilt\Lambda_n=n!$.
\end{corollary}

\begin{proof} It is not difficult to show that $\Gamma_{n-1}\simeq \Lambda_n/(e)$ by Example \ref{4.2}. Then by Theorem \ref{3.5} $\tilt_{1}\Lambda_{n}\simeq\sttilt\Gamma_{n-1}$ holds. On the other hand, $\sttilt\Gamma_{n-1}\simeq\mathfrak{S}_{n}$ holds by \cite[Theorem 0.1]{Mi}. We are done.
\end{proof}
Recall that an algebra is called {\it $\tau$-tilting finite} if
there are finite number of basic $\tau$-tilting modules up to
isomorphism. Applying Theorem \ref{3.5} to the Auslander algebra of
a Nakayama hereditary algebra, we have the following corollary.
\begin{corollary}\label{4.5}Let $\Lambda_{n}$ be the Auslander algebra of the Nakayama
hereditary algebra $KQ$ with $Q=A_n$. Then there is a bijection
$\tilt_1{\Lambda_n}\simeq\sttilt\Lambda_{n-1}$.
Thus $\Lambda$ is $\tau$-tilting finite if and only if
$n\leq 4$.
\end{corollary}

\begin{proof} It is
not difficult to show that $\Lambda_{n}/(e)\simeq\Lambda_{n-1}$.
Then by Theorem \ref{3.5}, $\tilt
\Lambda_n\simeq\sttilt\Lambda_{n-1}$. On the other hand, it was showed in \cite{K} $\Lambda_n$ has finite number of classical tilting modules if and only if $n\leq 5$. Then the assertion holds.
\end{proof}

 For more details of $\tau$-rigid modules over Auslander algebras, we
refer to \cite{Z}. Recall from \cite[Proposition 1.17]{I} that a
hereditary algebra is 1-Gorenstein if and only if it is a Nakayama
algebra. Now we have the following corollary.

\begin{corollary}\label{3.8} Let $\Lambda_n=KQ$ be the Nakayama hereditary algebra with
$Q=A_n$.
 Then there are bijections
 \[\tilt_1{\Lambda_n}\simeq\sttilt\Lambda_{n-1}\simeq
\{\mbox{clusters of the cluster algebra of type $A_{n-1}$}\}.\]
Thus $\#\sttilt\Lambda_n=(2(n+1))!/((n+2)!(n+1)!)$.
\end{corollary}
\begin{proof}
A straight calculation shows that
$\Lambda_{n-1}\simeq\Lambda_n/(e)$. Then by Theorem \ref{3.5} and
\cite[Theorem 1]{K}, $\tilt_1{\Lambda_n}\simeq\sttilt\Lambda_{n-1}$ and $\#\sttilt\Lambda_n=(2(n+1))!/((n+2)!(n+1)!)$ hold. By \cite[Theorem 0.5]{AIR} and \cite[Theorem 4.5]{BMRRT}, one gets the second bijection.
\end{proof}

{\bf Acknowledgement} Parts of the work were done when the second
author visited Nagoya in the year 2015. The second author wants to
thank the first author for hospitality during the stay in Nagoya. He
also wants to thank other people in Nagoya for their kind help.
Parts of the results were presented in USTC, Tsinghua University,
Anhui University and Capital Normal University. The second author would like to
thank all the hosts there for hospitality. The authors thank the referee for useful suggestions to improve this paper.

\end{document}